\documentclass[10pt]{article}

\usepackage[utf8]{inputenc}

\usepackage{url}

\usepackage{breakurl}
\usepackage[hidelinks,breaklinks]{hyperref}

\usepackage{geometry}
\usepackage[english]{babel}
\usepackage[backend=biber,style=alphabetic]{biblatex}
\usepackage{mathtools}
\usepackage{amsmath}
\usepackage{amssymb}
\usepackage{mathrsfs}
\usepackage{dsfont}
\usepackage{graphicx}
\usepackage{booktabs}
\usepackage{array}
\usepackage{paralist}
\usepackage{verbatim}
\usepackage{subfig}
\usepackage{tikz}
\usepackage{tikz-cd}
\usetikzlibrary{arrows}
\tikzset{node distance=1.8cm, auto}
\usepackage{csquotes}

\usepackage{lmodern}
\usepackage{sectsty}
\usepackage{amsthm}
\theoremstyle{definition}
\allsectionsfont{\sffamily\mdseries\upshape} 
\newtheorem{theorem}{Theorem}[section]
\newtheorem{proposition}[theorem]{Proposition}
\newtheorem{corollary}[theorem]{Corollary}
\newtheorem{definition}[theorem]{Definition}

\newtheorem{example}[theorem]{Example}

\newtheorem{lemma}[theorem]{Lemma}
\newtheorem{remark}[theorem]{Remark}
\newtheorem{construction}[theorem]{Construction}

\usepackage{rotating}
\usepackage{graphicx}

\let\phi\varphi
\let\theta\vartheta
\DeclareMathOperator\Spec{Spec}

\DeclareMathOperator\Hom{Hom}
\DeclareMathOperator\coker{coker}

\DeclareMathOperator\End{End}
\DeclareMathOperator\Aut{Aut}

\DeclareMathOperator\Gal{Gal}

\DeclareMathOperator\Tr{Tr}

\DeclareMathOperator\Sym{Sym}
\DeclareMathOperator\Frac{Frac}
\renewcommand{\:}{\colon}
\newcommand{\T}{\mathbf{T}}
\renewcommand{\S}{\operatorname{S}}

\newcommand{\bF}{\mathbb{F}}
\newcommand{\bG}{\mathbb{G}}

\newcommand{\cC}{\mathcal{C}}
\newcommand{\cD}{\mathcal{D}}

\newcommand{\cF}{\mathcal{F}}

\newcommand{\cM}{\mathcal{M}}

\title{\textsc{A Ramanujan bound for \\ Drinfeld modular forms}} 
\author{Sjoerd de Vries}
\date{} 

\begin{document}
\maketitle

\begin{abstract}
    \noindent We prove a Lefschetz trace formula for B\"ockle--Pink crystals on tame Deligne--Mumford stacks of finite type over~$\bF_q$ and apply it to the crystal associated to the universal Drinfeld module. Combined with the Eichler--Shimura theory developed by B\"ockle, this leads to a trace formula for Hecke operators on Drinfeld modular forms. As an application, we deduce a Ramanujan bound on the traces of Hecke operators.
\end{abstract}

\tableofcontents

\section{Introduction}
Drinfeld modules of rank 2 are function field analogues of elliptic curves. Since their inception due to Drinfeld \cite{drinf_ell}, much of the theory of elliptic curves has been translated to the function field setting. In particular, there is a theory of Drinfeld modular forms \cite{goss_pi-adic}, which are defined as sections of certain line bundles on moduli spaces of Drinfeld modules. In \cite{bockle}, B\"ockle develops an Eichler--Shimura theory for Drinfeld modular forms, showing that one can realise spaces of Drinfeld cusp forms as the compactly supported cohomology of certain crystals.

In the classical setting, Eichler--Shimura theory naturally leads to a proof of the Ramanujan bound, which states the following. Let $\Gamma \subseteq \operatorname{SL}_2(\mathbb Z)$ be a congruence subgroup of level~$N$, let~$p$ be a prime number not dividing~$N$, and let~$\alpha$ be an eigenvalue of the Hecke operator~$\T_p$ acting on the space $\S_k(\Gamma)$ of cusp forms of weight~$k$ and level~$\Gamma$. Then $|\alpha| \leq 2p^{(k-1)/2}$,
where $|\cdot|$ denotes the complex absolute value. Equivalently, for all~$n \geq 1$, we have
\[
|\Tr(\T^n_p \hspace{0.2em} | \hspace{0.1em} \S_k(\Gamma))| \leq \dim \S_k(\Gamma) \cdot 2^n p^{n(k-1)/2}.
\]

This bound was conjectured in~1916 by Ramanujan in the case $k = 12$ and $\Gamma = \operatorname{SL}_2(\mathbb Z)$. Almost six decades later, it became a celebrated result of Deligne \cite{deligne_ramanujan,deligne_weil}. Let us sketch the idea of the proof. By Eichler--Shimura theory, cusp forms arise in the cohomology of certain $\ell$-adic local systems on moduli spaces of elliptic curves, and the Hecke operator $\T_p$ acts locally at~$p$ as the sum of Frobenius and Verschiebung. Deligne's idea was to prove the Ramanujan bound by combining Eichler--Shimura theory with the purity of cohomology which follows from the Weil conjectures (more specifically, the Riemann hypothesis).

In the function field setting, an analogous Ramanujan bound has previously been observed \cite{bockle,li-meemark,nicole-rosso,band-val_atkin}, but due to the lack of purity for cohomology of crystals, a proof of this fundamental inequality has remained elusive. In this paper, we explain how the purity argument can be circumvented by using the Lefschetz trace formula. The moduli interpretation of the points on modular curves then allows us to deduce the Ramanujan bound for Drinfeld modular forms from the Riemann hypothesis for Drinfeld modules. The statement of the Ramanujan bound in level~1 is as follows:

\begin{theorem}[Corollary~\ref{cor:rambound_tr}]\label{thm:intro1}
    Let $\S_{k,l}$ denote the space of Drinfeld cusp forms of weight~$k \geq 2$ and type~$l \in \mathbb Z$. Then for any prime $\mathfrak p \trianglelefteq A$ with residue field of size~$\mathcal P$ and any $n \geq 1$, we have
    \[
    |\Tr (\T_{\mathfrak p}^n \hspace{0.1em} | \hspace{0.1em} \S_{k,l})|_\infty \leq \mathcal P^{n\left(\frac{k}{2} + l - k\right)}.
    \]
\end{theorem}

Along the way, we prove several results of independent interest. For good coefficient rings~$B$ \cite[Def. 9.7.4]{BP}, we define $l$-series of $B$-crystals on tame Deligne--Mumford stacks of finite type over~$\bF_q$ 
and show that the following version of the Lefschetz trace formula holds:

\begin{theorem}[Theorem~\ref{thm:trace_formula}]\label{thm:intro2}
    Let $f\: \mathfrak Y \to \mathfrak X$ be a compactifiable morphism of tame Deligne--Mumford stacks of finite type over~$\mathbb F_q$, let~$B$ be reduced, and let~$\underline{\mathcal F}^\bullet$ be a bounded complex of flat $B$-crystals on~$\mathfrak Y$. Then we have
\[
l(\mathfrak Y,\underline{\mathcal F}^\bullet,t) = l(\mathfrak X, Rf_!\underline{\mathcal F}^\bullet,t).
\]
\end{theorem}

Theorem~\ref{thm:intro2} in combination with B\"ockle--Eichler--Shimura theory leads to our third main result, which is a trace formula for Hecke operators on Drinfeld cusp forms.

\begin{theorem}[Theorem~\ref{thm:heckesum}]\label{thm:intro3}
    Let $\S_{k,l}$ denote the space of Drinfeld cusp forms of weight~$k$ and type~$l$, and let $\T_{\mathfrak p}$ be the Hecke operator associated to a prime~$\mathfrak p$ of degree~$d$. Then for every $n \geq 1$, we have
    \[
    \Tr(\T^n_{\mathfrak p} \hspace{0.2em} | \hspace{0.1em} \S_{k+2,l}) = \sum_{[(E,\phi)]/\bF_{{\mathfrak p}^n}} \sum_{i=0}^{k} \pi_{\phi}^{i+l-k-1} \bar{\pi}_{\phi}^{l-1-i},
    \]
    where the first sum is taken over the set of isomorphism classes of Drinfeld modules over~$\bF_{q^{dn}}$ with characteristic~$\mathfrak p$, $\pi_{\phi}$ denotes the Frobenius endomorphism of $(E,\phi)$, and $\bar{\pi}_{\phi}$ its conjugate.
\end{theorem}


This paper grew out of the author's licentiate thesis \cite{diva}.

 \subsection*{Outline of the paper}
 In Section~\ref{sec:moduli}, we recall some theory of Drinfeld modules and crystals and extend it to Deligne--Mumford stacks. In Section~\ref{sec:trace_formula}, we define $l$-series of crystals and prove the Lefschetz trace formula. In Section~\ref{sec:ramanujan}, we deduce the trace formula for Hecke operators on Drinfeld modular forms, as well as the Ramanujan bound.

 \subsection*{Notation and conventions}
 Throughout, $p$ denotes a prime number and $q$ a positive power of~$p$. The set of closed points of a scheme~$X$ is denoted by~$|X|$. We denote by~$C$ a smooth, geometrically connected, proper curve over $\bF_q$ with function field~$K = k(C)$. We fix a closed point $\infty \in |C|$ and denote the corresponding valuation on~$K$ by~$v_\infty$. We let $A := \mathcal O_C(C \setminus \{\infty\})$ denote the ring of integers in~$K$. 
 
 For a maximal ideal $\mathfrak p \trianglelefteq A$, we denote by $\mathbb F_{\mathfrak p}$ its residue field, and by $\mathbb F_{\mathfrak p^n}$ its unique degree~$n$ extension. The symbol $\mathfrak p$ never denotes the zero ideal. 
 
 If $k$ is a field, we denote by $\bar{k}$ its algebraic closure. If $V$ is a vector space, we denote by $V^\vee$ its dual. For a ring~$R$, we denote by $\text{Nil}(R)$ the nilradical of~$R$, by $\textsf{Sch}_R$ the category of $R$-schemes, and by $\text{Cent}_R(r)$ the centralizer of an element $r \in R$.
 
 All algebraic stacks are assumed to be noetherian and separated.

\subsection*{Acknowledgements}
I would like to thank Dan Petersen and David Rydh for helpful conversations, and Lucas Mann for answering some questions about his work. I would also like to thank my supervisor Jonas Bergstr\"om and my co-supervisor Olof Bergvall.

\section{Drinfeld modules, \texorpdfstring{$\tau$}{t}-sheaves, and crystals}\label{sec:moduli}
In this section, we recall the basic theory of Drinfeld modules over schemes. We then define the moduli stack of Drinfeld $A$-modules and the category of crystals on a Deligne--Mumford stack, along with its pre-6-functor formalism.

For more thorough treatments of the theory of Drinfeld modules, we refer the reader to~\cite{goss_book,papikian}. For background material on stacks, we refer the reader to~\cite{olsson}. The theory of crystals on schemes developed in~\cite{BP}. The extension to stacks is based on the results from~\cite{mann_thesis}.

\subsection{Drinfeld modules}

Let $S$ be an $\bF_q$-scheme. By a line bundle $E/S$, we mean a commutative $S$-group scheme which is Zariski-locally on $S$ isomorphic to $\mathbb G_{a}$. Denote by $\End_S^{\mathbb F_q}(E)$ the ring of $\bF_q$-linear $S$-group scheme endomorphisms of~$E$. Locally, this ring can be understood as follows.

\begin{proposition}\label{prop:Rtau}
    Let $S = \Spec(R)$ be an affine $\bF_q$-scheme. Denote by $R\{\tau\}$ the \discretionary{non-commu-}{tative}{non-commutative} polynomial ring in $\tau$ satisfying $\tau r = r^q \tau$ for all $r \in R$. Then the map
    \[
    R\{\tau\} \longrightarrow \End_S^{\mathbb F_q}(\bG_{a,S})
    \]
    sending $\tau$ to the $q$-Frobenius endomorphism is a ring isomorphism. \hfill $\Box$
\end{proposition}

\begin{definition}
    Let $S$ be an $\mathbb F_q$-scheme.
    
\textbf{1.} A \emph{Drinfeld $A$-module of rank $r \geq 1$ over $S$} is a pair $(E,\phi)$ consisting of a line bundle $E/S$ and an $\bF_q$-algebra homomorphism $\phi\: A \to \End_S^{\mathbb F_q}(E)$ with the following property: for any open subset $\Spec(R) = U \subseteq S$ trivialising $E$ and any $a \in A$, we have
\[
\phi(a)|_U = \sum_{i=0}^{n} \alpha^U_i(a) \tau^i \in R\{\tau\}
\]
with the following properties:
\begin{enumerate}
\item $\alpha^U_i(a) \in R^\times$ for $i = -r \cdot \deg(\infty) \cdot v_\infty(a)$;
\item $\alpha^U_i(a) \in \text{Nil}(R)$ for $i > -r \cdot \deg(\infty) \cdot v_\infty(a)$.
\end{enumerate}

\textbf{2.} A \emph{morphism of Drinfeld modules} $f \: (E,\phi) \to (E',\phi')$ is a morphism $E \to E'$ of $S$-group schemes such that for all $a \in A$, the following diagram commutes:
\[
\begin{tikzpicture}
\node(A){$E$};
\node at (2,0)(B){$E$};
\node[below of =A](C){$E'$};
\node[below of =B](D){$E'$};
\draw[->](A) to node{$\phi(a)$}(B);
\draw[->](A) to node[swap]{$f$}(C);
\draw[->](B) to node{$f$}(D);
\draw[->](C) to node{$\phi'(a)$}(D);
\end{tikzpicture}
\]
\end{definition}

\begin{definition}\label{def:char_morphism}
Let $(E/S,\phi)$ be a Drinfeld $A$-module. The \emph{characteristic morphism} $\vartheta \: S \to \Spec(A)$ of $\phi$ is defined via the map on sections
\[
\theta^\sharp \: A \xrightarrow{\ \phi \ } \End_S^{\mathbb F_q}(E) \xrightarrow{\ D \ } H^0(S,\mathcal O_S),
\]
where $D$ is the map which locally sends $\sum{\alpha_i\tau^i} \mapsto \alpha_0$.
\end{definition}

\begin{definition}
    Fix an integer $r \geq 1$. The \emph{moduli stack of Drinfeld $A$-modules of rank~$r$} is the category fibered in groupoids $\mathfrak M_r^A \to \mathbb F_q$ whose objects are
    \[
    \text{Ob}(\mathfrak M_r^A) = \left \{(E/S,\phi) \ | \ S \in \textsf{Sch}_{\mathbb F_q}, \ \phi\: A \to \End_S(E) \text{ is a Drinfeld $A$-module of rank }r \right \},
    \]
    and whose morphisms $(E'/S',\phi') \to (E/S,\phi)$ are given by pullback diagrams
    \[
    \begin{tikzpicture}
        \node(A){$E'$};
        \node(X) at (0.45,-0.45){\scalebox{1.5}{$\lrcorner$}};
        \node[right of =A](B){$E$};
        \draw[->](A) to (B);
        \node(C)[below of =A]{$S'$};
        \draw[->](A) to (C);
        \node(D)[below of =B]{$S$};
        \draw[->](B) to (D);
        \draw[->](C) to (D);
    \end{tikzpicture}
    \]
    in the category of Drinfeld modules.
\end{definition}

    We will often suppress the ring $A$ from the notation and simply write $\mathfrak M_r$ for the stack defined above. The characteristic morphism induces a structure map $\Theta \: \mathfrak M_r \to \Spec(A)$. The fibered category $\mathfrak M_r$ is representable by a smooth affine Deligne--Mumford stack of pure relative dimension $r-1$ over~$A$ \cite[Cor. 1.4.3]{laumon}, which we also denote by $\mathfrak M_r$. If $r=2$, it can be thought of as the function field analogue of the modular curve $\mathcal M_{1,1} \to \Spec(\mathbb Z)$.
    
    The following lemma shows that $\mathfrak M_r$ is tame, i.e., for every geometric point~$x$ of~$\mathfrak M_r$, the order of the automorphism group of~$x$ is prime to~$p$.

\begin{lemma}\label{lem:aut_order}
    Let $\phi$ be a Drinfeld $A$-module of rank~$r$ over a field~$k$ of characteristic~$p$. Then $\# \Aut(\phi) \equiv -1 \pmod{p}$.
\end{lemma}

\begin{proof}
    This follows because for some integers $1 \leq n_a \leq -r \deg(\infty) v_\infty(a)$, we have
    \[
    \Aut(\phi) = \bigcap_{a \in A} \text{Cent}_{k\{\tau\}}(\phi(a))^\times = \bigcap_{a \in A} \mathbb F_{q^{n_a}}^\times = \mathbb F_{q^{\gcd\{n_a | a \in A\}}}^\times.
    \]
\end{proof}

\begin{definition}
    A \emph{Drinfeld $A$-module of rank $r$} over a Deligne--Mumford stack $\mathfrak X$ is a morphism of stacks $\mathfrak X \to \mathfrak M_r^A$.
\end{definition}

\subsection{\texorpdfstring{$\tau$}{t}-sheaves and crystals}

Having defined Drinfeld modules over stacks, we now want to associate crystals to them. We start by defining $\tau$-sheaves and crystals over stacks.

\begin{definition}
    Let $B$ be an $\mathbb F_q$-algebra. We denote by $\textsf{Coh}_B \to \bF_q$ the category fibered in groupoids whose objects are
     \[
    \text{Ob}(\textsf{Coh}_B) = \left\{ (\mathcal F, S) \ | \ S \in \textsf{Sch}_{\bF_q}, \hspace{0.15em} \mathcal F \text{ is a coherent Zariski sheaf of } \mathcal O_{S \otimes B} \text{-modules} \right\},
    \]
    and whose morphisms are given by
    \[
    \Hom(\mathcal F'/S',\mathcal F/S) = \{ (s,t) \ | \ s\: S' \to S, \hspace{0.2em} t\: \mathcal F' \xrightarrow{\ \sim \ } (s\times \text{id})^*\mathcal F \}.
    \]
    We call $\textsf{Coh}_B$ the \emph{stack of coherent sheaves with coefficients in~$B$}. The category $\textsf{QCoh}_B \to \mathbb F_q$ of \emph{quasi-coherent sheaves with coefficients in~$B$} is defined as above with the word ``coherent" replaced by ``quasi-coherent".
\end{definition}

The categories $\textsf{Coh}_B$ and $\textsf{QCoh}_B$ are stacks (although they are not algebraic). Setting $B = \mathbb F_q$ recovers the usual stack of (quasi-)coherent sheaves. If $S$ is a scheme, an isomorphism class of maps $S \to \textsf{Coh}_B$ is equivalent to a coherent sheaf on $S \otimes B$.

For a Deligne--Mumford stack $\mathfrak X$, denote by $\sigma = \sigma_{\mathfrak X}$ the $q$-Frobenius endomorphism of $\mathfrak X$.

\begin{definition}
A \emph{(quasi-)coherent $\tau$-sheaf on~$\mathfrak X$ over~$B$} is a pair $\underline{\mathcal F} = (\mathcal F,\tau_{\mathcal F})$ consisting of the following data:
\begin{itemize}
    \item A (quasi-)coherent sheaf $\mathcal F \colon \mathfrak X \to \textsf{Coh}_B$ (resp.\ $\textsf{QCoh}_B$);
    \item An $\mathcal O_{\mathfrak X \otimes B}$-linear map $\tau_{\mathcal F} \colon \sigma^*\mathcal F \to \mathcal F$. 
\end{itemize}


A \emph{morphism of $\tau$-sheaves} $\underline{\mathcal F} \to \underline{\mathcal G}$ is a natural transformation $\psi \colon \mathcal F \to \mathcal G$ such that $\psi \circ \tau_{\mathcal F} = \tau_{\mathcal G} \circ \sigma^*\psi$. The category of coherent $\tau$-sheaves on~$\mathfrak X$ over~$B$ is denoted $\textsf{Coh}_\tau(\mathfrak X,B)$.
\end{definition}

\begin{definition}\label{def:nil-iso}
A $\tau$-sheaf $\underline{\mathcal F} \in \textsf{Coh}_\tau(\mathfrak X,B)$ is called \emph{nilpotent} if some power $\tau^n_{\mathcal F} \colon (\sigma^n)^*\mathcal F \to \mathcal F$ of $\tau_{\mathcal F}$ is zero. We denote the full subcategory of nilpotent coherent $\tau$-sheaves by $\textsf{NilCoh}_\tau(\mathfrak X,B)$. A morphism $f \: \underline{\mathcal F} \to \underline{\mathcal G}$ of $\tau$-sheaves is called a \emph{nil-isomorphism} if $\ker(f)$ and $\coker(f)$ are nilpotent $\tau$-sheaves.
\end{definition}

\begin{proposition}
Let $\mathfrak X$ be a Deligne--Mumford stack. Then $\textsf{NilCoh}_\tau(\mathfrak X,B)$ is a Serre subcategory of $\textsf{Coh}_\tau(\mathfrak X,B)$.
\end{proposition}

\begin{proof}
We need to show that $\textsf{NilCoh}_\tau$ is closed under taking subobjects, quotients, extensions, and isomorphisms. By \cite[Prop.\ 3.3.5]{BP}, this is true whenever $\mathfrak X$ is a scheme. We conclude by applying Corollary \ref{cor:tau_pullback} below.
\end{proof}

\begin{definition}
    The category $\textsf{Crys}(\mathfrak X,B)$ of \emph{$B$-crystals on~$\mathfrak X$} is the localisation of $\textsf{Coh}_\tau(\mathfrak X,B)$ at $\textsf{NilCoh}_\tau$.
\end{definition} 

More concretely, the objects in $\textsf{Crys}(\mathfrak X,B)$ are coherent $\tau$-sheaves on~$\mathfrak X$ over~$B$, and morphisms $\underline{\mathcal F} \to \underline{\mathcal G}$ are given by roofs, i.e., diagrams of the form
\[
\underline{\mathcal F} \longleftarrow \underline{\mathcal Z} \longrightarrow \underline{\mathcal G} 
\]
where $\underline{\mathcal F} \longleftarrow \underline{\mathcal Z}$ is a nil-isomorphism and $\underline{\mathcal Z} \longrightarrow \underline{\mathcal G}$ is any morphism of coherent $\tau$-sheaves.

\begin{remark}
    Every morphism of crystals $\underline{\mathcal F} \to \underline{\mathcal G}$ can be represented by a roof of the form
    \[
    \underline{\mathcal F} \xleftarrow{\hspace{0.1em} \tau^n \hspace{0.1em}} (\sigma^{n})^*\underline{\mathcal F} \longrightarrow \underline{\mathcal G};
    \]
    when $\mathfrak X$ is a scheme, this is \cite[Prop.\ 3.4.6]{BP}, and in general it follows from the descent results in \textsection \ref{sec:descent}. Thus, passing from $\tau$-sheaves to crystals is equivalent to inverting $\tau$. 
\end{remark}

One similarly defines $\textsf{QCrys}(\mathfrak X,B)$ as the localisation of $\textsf{QCoh}_\tau(\mathfrak X,B)$ at \emph{locally nilpotent} $\tau$-sheaves; see~\cite[\textsection 3.3]{BP}. This ensures that $\textsf{Crys}(\mathfrak X,B)$ is a full subcategory of~$\textsf{QCrys}(\mathfrak X,B)$.

\subsection{The \texorpdfstring{$\tau$}{t}-sheaf associated to a Drinfeld module}\label{sec:assoc-tau-sheaf}

Consider the map $\Hom(-,\mathbb G_a) \colon \mathfrak M^A_r \to \textsf{Coh}_A$, which acts on objects as
\[
(E/S,\phi) \longmapsto \mathcal{H}om_{S\textsf{-Grp}}^{\mathbb F_q}(E,\mathbb G_a),
\]
where the right-hand side denotes the coherent sheaf of $\mathcal O_{S \otimes A}$-modules of $\mathbb F_q$-linear $S$-group scheme morphisms $E \to \mathbb G_a$, where $a \in A$ acts via right multiplication by~$\phi(a)$. The map on morphisms is induced by precomposition. By Proposition~\ref{prop:Rtau}, the coherent sheaf $\mathcal{H}om_{S\textsf{-Grp}}^{\mathbb F_q}(E,\mathbb G_a)$ is Zariski-locally isomorphic to $\mathcal O_S\{\tau\}$, which is locally free of rank $r$ over $\mathcal O_{S \otimes A}$ \cite[Prop.~3]{drinf_comm-subrings}.



\begin{construction}\label{constr:ass-tau}
    Given a Drinfeld module $(\mathcal E,\phi)\: \mathfrak X \to \mathfrak M_r$, we associate a $\tau$-sheaf $\underline{\cM}(\phi)$ to it as follows: the underlying sheaf is the composition
    \[
        \mathcal M(\phi): \mathfrak X \xrightarrow{(\mathcal E,\phi)} \mathfrak M_r \xrightarrow{\Hom(-,\mathbb G_a)} \textsf{Coh}_A,
    \]
    and the map $\tau = \tau_{\mathcal M(\phi)}$ is the natural transformation $\sigma^* \mathcal M(\phi) \to \mathcal M(\phi)$ given by composition with $\sigma_{\mathbb G_a}$ on the left, where $\sigma$ denotes the $q$-Frobenius endomorphism.
\end{construction}

\begin{example}\label{ex:Mphi}
    Suppose $\mathfrak X = \Spec(k)$ for some finite extension $k = \mathbb F_{q^n}$ of $\bF_q$. Then $\mathcal E \cong \mathbb G_a$, and the underlying coherent sheaf of~$\cM(\phi)$ is given by $\End_{k\textsf{-Grp}}^{\mathbb F_q}(\mathbb G_a) \cong k\{\tau\}$, cf.\ Proposition~\ref{prop:Rtau}. The left action of $\alpha \otimes a \in k \otimes A$ on $\psi \in k\{\tau\}$ is given by $(\alpha \otimes a) \star \psi = \alpha \psi \phi(a)$. 
    The morphism $\tau\: \sigma^* k\{\tau\} \to k\{\tau\}$ is given by multiplication by $\tau$ on the left, which is indeed a module homomorphism by the commutation relation in $k\{\tau\}$. 
    We conclude this example by noting that the element~$\tau^n$ is the Frobenius endomorphism of~$(\mathcal E,\phi)$.
\end{example}

\subsection{Descent}\label{sec:descent}

A (quasi-)coherent sheaf on a Deligne--Mumford stack is equivalent to a (quasi-)coherent sheaf with a descent datum on an \'etale cover. In this subsection, we prove that the same is true for crystals. Throughout, $\mathfrak X$ denotes a Deligne--Mumford stack over~$\bF_q$, and $\pi\: X \to \mathfrak X$ denotes an \'etale cover by a scheme.

\begin{definition}
    Let $\textsf{C}(-) \in \{\textsf{Coh}_\tau(-,B), \textsf{QCoh}_\tau(-,B), \textsf{Crys}(-,B), \textsf{QCrys}(-,B)\}$. Define the \emph{category of descent data} $\textsf{C}(X \to \mathfrak X)$ as follows: objects are pairs $(\underline{\cF},\psi)$ consisting of an object $\underline{\cF} \in \textsf{C}(X)$ and a descent datum $\psi \: \text{pr}_1^*\underline{\cF} \xrightarrow{\sim} \text{pr}_2^*\underline{\cF}$ in $\Aut_{\textsf{C}}(X \times_{\mathfrak X} X)$ which satisfies the cocycle condition: i.e., we require that the diagram
    \[
        \begin{tikzpicture}
            \node(A){$\text{pr}_{1,2}^*\text{pr}_1^*\underline{\cF}$};
            \node(B) at (1.8,1.5){$\text{pr}_{1,3}^*\text{pr}_1^*\underline{\cF}$};
            \node(C) at (5,1.5){$\text{pr}_{1,3}^*\text{pr}_2^*\underline{\cF}$};
            \node(D) at (6.8,0){$\text{pr}_{2,3}^*\text{pr}_2^*\underline{\cF}$};
            \node(E) at (1.8,-1.5){$\text{pr}_{1,2}^*\text{pr}_2^*\underline{\cF}$};
            \node(F) at (5,-1.5){$\text{pr}_{2,3}^*\text{pr}_1^*\underline{\cF}$};
            \draw [double equal sign distance] (A) to (B);
            \draw [double equal sign distance] (C) to (D);
            \draw [double equal sign distance] (E) to (F);
            \draw[->](B) to node{$\text{pr}_{1,3}^*\psi$}(C);
            \draw[->](A) to node{$\text{pr}_{1,2}^*\psi$}(E);
            \draw[->](F) to node{$\text{pr}_{2,3}^*\psi$}(D);
        \end{tikzpicture}
    \]
    commutes. A morphism $(\underline{\cF},\psi) \to (\underline{\cF}',\psi')$ is a map $f\: \underline{\cF} \to \underline{\cF}'$ in $\textsf{C}(X)$ such that $\text{pr}_2^*f \circ \psi = \psi' \circ \text{pr}_1^*f$.
\end{definition}

Note that $\pi^*$ induces a functor from $\textsf{C}(\mathfrak X)$ to $\textsf{C}(X \to \mathfrak X)$. The goal of this section is to show this is an equivalence in all cases; in other words, $\tau$-sheaves and crystals satisfy effective descent with respect to the \'etale topology.

\begin{proposition}[Descent for $\tau$-sheaves]\label{prop:tau_descent}
    Let $\textsf{C}(-) \in \{\textsf{Coh}_\tau(-,B), \textsf{QCoh}_\tau(-,B)\}$. Then $\pi^*$ induces an equivalence of categories $\textsf{C}(\mathfrak X) \xrightarrow{\sim} \textsf{C}(X \to \mathfrak X)$.
\end{proposition}
\begin{proof}
    This follows from descent for (quasi-)coherent sheaves and unwinding definitions. In particular, a morphism $\tau_{\cF} \colon \sigma^*\cF \to \cF$ of (quasi-)coherent sheaves on $X$ descends to $\mathfrak X$ if and only if it is compatible with the descent data $\psi$ and $\sigma^*\psi$, which is precisely the condition that $\psi$ is a morphism of $\tau$-sheaves. Similarly, a morphism of $\tau$-sheaves on $X$ descends to $\mathfrak X$ if and only if it is compatible with the descent data.
\end{proof}


\begin{corollary}\label{cor:tau_pullback}
    The pullback functor $\pi^*\colon \textsf{Coh}_\tau(\mathfrak X,B) \to \textsf{Coh}_\tau(X,B)$ is exact and conservative. Moreover, $\underline{\mathcal F}$ is nilpotent if and only if $\pi^*\underline{\mathcal F}$ is nilpotent.
\end{corollary}
\begin{proof}
    Using the equivalence from Proposition \ref{prop:tau_descent}, we can describe $\pi^*$ as the functor sending $(\underline{\cF},\psi) \mapsto \underline{\cF}$. Then $\pi^*$ is exact because forgetting the descent datum commutes with taking kernels and cokernels. It is conservative because the property of a morphism of $\tau$-sheaves on $\mathfrak X$ being an isomorphism does not depend on the descent datum. Similarly, a $\tau$-sheaf is nilpotent if and only if $\tau^n = 0$ for some $n \geq 1$, which does not depend on the descent datum.
\end{proof}

\begin{construction}\label{constr:perf}
Let $\underline{\mathcal F}$ be a coherent $\tau$-sheaf on~$\mathfrak X$. Iterated composition of the adjoint of~$\tau$ yields a direct system
\[
   \sigma_*^\bullet \tau := \left( \mathcal F \longrightarrow \sigma_* \mathcal F \longrightarrow \sigma^2_* \mathcal F \longrightarrow \ldots \ \right)
\]
Write $\widehat{\mathcal F} := \varinjlim \sigma_*^\bullet \tau$ for the direct limit. The adjoint of the canonical isomorphism $\widehat{\mathcal F} \to \sigma_*\widehat{\mathcal F}$ gives~$\widehat{\mathcal F}$ the structure of a $\tau$-sheaf, which we denote by~$\underline{\widehat{\mathcal F}}$. We call~$\underline{\widehat{\mathcal F}}$ the \emph{perfection} of~$\underline{\mathcal F}$. The assignment $\underline{\mathcal F} \mapsto \underline{\widehat{\mathcal F}}$ yields a functor $\text{perf} \: \textsf{QCoh}_\tau(\mathfrak X,B) \to \textsf{QCoh}_\tau(\mathfrak X,B)$.
\end{construction}

One easily sees that $\underline{\widehat{\mathcal F}} = 0$ if $\underline{\mathcal F}$ is nilpotent. More generally, if $f$ is a nil-isomorphism, then $\text{perf}(f)$ is an isomorphism.

\begin{lemma}\label{lem:perf}
    The induced functor $\overline{\text{perf}}\: \textsf{QCrys}(\mathfrak X,B) \to \textsf{QCoh}_\tau(\mathfrak X,B)$ is an exact, fully faithful right adjoint to the localisation functor.
\end{lemma}

\begin{proof}
    Exactness follows from the exactness of filtered direct limits in $\textsf{QCoh}_\tau(\mathfrak X,B)$, cf.\ \cite[Tag 0781]{stacks_project}. The rest is formal: the adjunction is proved as in \cite[Prop.\ 3.4.8]{BP}, and the counit is an isomorphism by \cite[Prop.\ 3.3.13]{BP}; the latter is equivalent to the right adjoint being fully faithful.
\end{proof}

\begin{proposition}[Descent for crystals]
    Let $\textsf{C}(-) \in \{\textsf{Crys}(-,B), \textsf{QCrys}(-,B)\}$. Then $\pi^*$ induces an equivalence of categories $\textsf{C}(\mathfrak X) \xrightarrow{\sim} \textsf{C}(X \to \mathfrak X)$.
\end{proposition}

\begin{proof}

We first extend the functor $\overline{\text{perf}}$ to the categories of descent data. Note that for $i \in \{1,2\}$, we have a natural isomorphism $\widehat{\text{pr}^*_i \underline{\cF}} \xrightarrow{\sim} \text{pr}^*_i \widehat{\underline{\cF}}$, since $\text{pr}_i \: X \times_{\mathfrak X} X \to X$ is flat and hence $\text{pr}_i^*\sigma_* \cong \sigma_*\text{pr}_i^*$. It is then clear that a descent datum $\psi$ of objects in $\textsf{C}$ induces a descent datum $\hat{\psi}$ of their perfections, and that a morphism $f$ in $\textsf{C}(X \to \mathfrak X)$ induces a morphism $\hat{f}$ in $\textsf{QCoh}_\tau(X \to \mathfrak X,B)$. 
Since the perfection functor is fully faithful on $\textsf{C}(X)$ and $\textsf{C}(X \times_{\mathfrak X} X)$ by Lemma~\ref{lem:perf},
so is the perfection functor on~$\textsf{C}(X \to \mathfrak{X})$.

Now consider the diagram
   \[
   \begin{tikzpicture}
       \node(A){$\textsf{C}(\mathfrak X)$};
       \node(B) at (4,0) {$\textsf{QCoh}_\tau(\mathfrak X,B)$};
       \node(C)[below of =A]{$\textsf{C}(X \to \mathfrak X)$};
       \node(D)[below of =B]{$\textsf{QCoh}_\tau(X \to \mathfrak X,B)$};
       \node(E) at (8.5,0) {$\textsf{QCrys}(\mathfrak X,B)$};
       \node(F)[below of =E] {$\textsf{QCrys}(X \to \mathfrak X,B)$};
       \draw[->](B) to node{$q$}(E);
       \draw[->](D) to node{$q$}(F);
       \draw[->](E) to (F);
       \draw[->](A) to node{$\overline{\text{perf}}$}(B);
       \draw[->](C) to node{$\overline{\text{perf}}$}(D);
       \draw[->](A) to node{$\pi^*$}(C);
       \draw[->](B) to node{$\pi^*_{\textsf{Q}}$}(D);
   \end{tikzpicture}
   \]
where $q$ denotes localisation. By the above, the perfection functors are fully faithful. Since $\pi_\textsf{Q}^*$ is an equivalence by Proposition~\ref{prop:tau_descent}, $\pi^*$ is also fully faithful. Moreover, we have a natural isomorphism $q \circ \overline{\text{perf}} \cong \text{id}$ by Lemma~\ref{lem:perf}. Hence a simple diagram chase, along with the fact that $\textsf{C}(X \to \mathfrak X)$ is a full subcategory of $\textsf{QCrys}(X \to \mathfrak X,B)$, shows that $\pi^*$ is essentially surjective and hence an equivalence.
\end{proof}

We record the following consequence, which is proved in the same way as Corollary~\ref{cor:tau_pullback}.

\begin{corollary}\label{cor:crys_pullback}
    The functor $\pi^* \colon \textsf{Crys}(\mathfrak X,B) \to \textsf{Crys}(X,B)$ is exact and conservative. \hfill \qedsymbol{}
\end{corollary}

\subsection{Derived categories of crystals}
An important aspect of the theory of crystals are the functors $f^*, f_!$ and $\otimes$. Together, these establish a pre-6-functor formalism in the sense of \cite[Appendix A.5]{mann_thesis}.
Our primary goal is to apply the theory of crystals to the moduli stack~$\mathfrak M_r$. For this reason, we study the bounded derived category $D^b(\textsf{Crys}(\mathfrak X,B))$
of crystals on a tame Deligne--Mumford stack $\mathfrak X$. One could mimic the constructions from~\cite{BP} to construct the derived functors on $D^b(\textsf{Crys}(\mathfrak X,B))$ directly. Instead we will, where possible, apply the results of \cite{mann_thesis} to construct the functors abstractly.

Recall that a morphism $f \: \mathfrak Y \to \mathfrak X$ is called \emph{compactifiable} if it can be factored as $f = \overline{f} \circ j$, where $j$ is an open immersion and $\overline{f}$ is proper.

\begin{theorem}\label{thm:functors}
    Let $f\: \mathfrak Y \to \mathfrak X$ be a morphism between Deligne--Mumford stacks of finite type over~$\bF_q$, and let $B \to B'$ be a ring map. Then we have the following functors:
    \begin{align*}
        f^*&\: D^b(\textsf{Crys}(\mathfrak X,B)) \longrightarrow D^b(\textsf{Crys}(\mathfrak Y,B)); \\
        - \otimes^L - &\: D^-(\textsf{Crys}(\mathfrak X,B)) \times D^-(\textsf{Crys}(\mathfrak X,B)) \longrightarrow D^-(\textsf{Crys}(\mathfrak X,B)); \\
        - \otimes^L_B B' &\: D^b(\textsf{Crys}(\mathfrak X,B)) \longrightarrow D^b(\textsf{Crys}(\mathfrak X,B')).
    \end{align*}
    If $f$ is proper, we also have a derived pushforward
    \[
    Rf_*\: D^b(\textsf{Crys}(\mathfrak Y,B)) \longrightarrow D^b(\textsf{Crys}(\mathfrak X,B)),
    \]
    and if $f$ is compactifiable, we have a proper pushforward
    \[
    Rf_! \: D^b(\textsf{Crys}(\mathfrak Y,B)) \longrightarrow D^b(\textsf{Crys}(\mathfrak X,B)).
    \]
    These functors satisfy all usual compatibilities, including the base change isomorphism and projection formula.
\end{theorem}

\begin{proof}
    By \cite[Chapter 6]{BP}, the theorem is true when $f$ is a morphism of schemes. We use \cite[A.5]{mann_thesis} to extend it to representable morphisms of stacks, as follows. With notation as in loc.\ cit., we let $(\mathcal C,E)$ be the geometric setup given by $\mathcal C = \textsf{Sch}_{\bF_q}$ and $E$ is the class of compactifiable morphisms. One easily checks that the pair $(I,P)$, where~$I$ is the class of open immersions and~$P$ is the class of proper morphisms, forms a suitable decomposition of~$E$ as in Def.\ A.5.9 of loc.\ cit. Let $\cD\: \cC^{\text{op}} \to \textsf{Cat}_\infty^\otimes$ be the functor sending $X \mapsto D^-(\textsf{Crys}(X,B))$ and $f \mapsto f^*$. Then Prop.\ A.5.10 in loc.\ cit.\ yields a pre-6-functor formalism $\mathcal D \colon \text{Corr}(\mathcal C)_{E,all} \to \textsf{Cat}_\infty$. One can do the same with $\mathcal D(X)$ replaced with $D^b(\textsf{Crys}^{\text{flat}}(X,B))$ (cf.\ Definition ~\ref{def:flat}), by \cite[Chapter 7]{BP}, or with $D^-(\textsf{QCrys}(X,B))$.
    
    Let $\mathcal C'$ be the category of Deligne--Mumford stacks over~$\bF_q$, and let $E'$ denote the class of compactifiable morphisms which are representable by schemes. Then \cite[Prop.\ A.5.16]{mann_thesis} yields the desired pre-6-functor formalism for crystals on stacks, which automatically satisfies the claimed compatibilities. The only difference is that the functors thus obtained are defined on the \emph{bounded above} derived category. To see that boundedness is preserved by $f^*$ and $Rf_!$, we apply Corollary~\ref{cor:crys_pullback}. It implies that if $\pi \colon X \to \mathfrak X$ is an \'etale cover and $\underline{\cF}^\bullet \in D^-(\textsf{Crys}(\mathfrak X,B))$ satisfies $\pi^*\underline{\cF}^\bullet \in D^b(\textsf{Crys}(X,B))$, then $\underline{\cF}^\bullet \in D^b(\textsf{Crys}(\mathfrak X,B))$; combining this with proper base change yields the claim. Note that so far, $Rf_!$ is only defined for representable $f$.
    
    To define $Rf_!$ for compactifiable morphisms, it suffices to define $Rf_*$ for proper $f$. To see this, consider a decomposition $f = \bar{f} \circ j$ with $\bar{f}$ proper and $j$ an open immersion. Then we may define $Rf_! := R\bar{f}_* \circ j_!$, noting that $Rj_! = j_!$ is exact. By the same arguments as in \cite[\textsection 6.7]{BP}, the assignment $f \mapsto Rf_!$ is then natural in $f$ and does not depend on the chosen compactification; moreover, if $f$ is representable by schemes, this definition of $Rf_!$ coincides with the previous one.
    
    To define $Rf_*$ for proper $f$, note that this functor already exists on the level of coherent sheaves \cite[Thm 11.6.1]{olsson}.
    It can be extended to $\tau$-sheaves by virtue of the base change morphism $\sigma^*f^* \to f^*\sigma^*$. This functor $Rf_*$ on $\tau$-sheaves preserves nilpotence, so it induces a functor $Rf_*$ on crystals which behaves as one expects.
    
    The last functor $- \otimes^L_B B'$ falls outside of the scope of \cite{mann_thesis}, but is simple to describe: it is induced by the extension of scalars functor on the underlying sheaves, and is straightforward to construct for crystals on stacks. This completes the proof.
\end{proof}
\section{Lefschetz trace formula}\label{sec:trace_formula}
In this section, we define $l$-series of flat crystals on tame Deligne--Mumford stacks. In the case of schemes, the $l$-series is the logarithmic derivative of the crystalline $L$-function defined in \cite{BP}. Our main result is a Lefschetz trace formula which relates the $l$-series of the fibers of a crystal to the $l$-series of its compactly supported cohomology, cf.\ Theorem~\ref{thm:trace_formula}.

\subsection{Flatness}
Let $\mathfrak X$ be a Deligne--Mumford stack of finite type over~$\bF_q$.
In order to begin discussing $l$-series of $B$-crystals, one needs to isolate a class of crystals for which it makes sense to talk about the trace of~$\tau$ when considered as a $B$-linear map. This leads to the notion of flat crystals. 

\begin{definition}\label{def:flat}
    A crystal $\underline{\cF} \in \textsf{Crys}(\mathfrak X,B)$ is \emph{flat} if the functor $\underline{\mathcal F} \otimes - \colon \textsf{Crys}(\mathfrak X,B) \to \textsf{Crys}(\mathfrak X,B)$ is exact. The full subcategory of flat crystals is denoted by $\textsf{Crys}^\text{flat}(\mathfrak X,B)$.
\end{definition}

Flatness is preserved by all the derived functors from Theorem~\ref{thm:functors},
and the derived tensor product $\underline{\cF}^\bullet \otimes^L \underline{\mathcal G}^\bullet$ of two bounded complexes $\underline{\cF}^\bullet, \underline{\mathcal G}^\bullet \in D^b(\textsf{Crys}^\text{flat}(\mathfrak X,B))$ lies again in $D^b(\textsf{Crys}^\text{flat}(\mathfrak X,B))$.

The following result shows that flatness is a pointwise property, and hence the theory of flat crystals over stacks is analogous to the theory for schemes.

\begin{lemma}\label{lem:flat}
    A crystal $\underline{\cF} \in \textsf{Crys}(\mathfrak X,B)$ is flat if and only if $x^*\underline{\cF}$ is flat for every $x \in \mathfrak X(\bF_{q^n})$.
\end{lemma}

\begin{proof}
    Let $\pi \: X \to \mathfrak X$ be an \'etale cover by a scheme. By Corollary \ref{cor:crys_pullback} and the fact that pullback commutes with tensor product, $\underline{\cF}$ is flat if and only if $\pi^*\underline{\cF}$ is flat.
    Since moreover every $x \: \Spec(\bF_{q^n}) \to \mathfrak X$ lifts to a closed point of~$X$ after a finite field extension, the result follows from \cite[Prop.\ 7.2.6 and Cor.\ 4.6.3]{BP}.
\end{proof}

The following result is \cite[Prop.\ 9.3.4]{BP}. It implies that flatness is related to being pointwise locally free, and will allow us to define crystalline $L$-functions.

\begin{proposition}\label{prop:art-perf}
    Suppose $B$ is artinian, and let $k/\bF_q$ be finite. Let $\underline{\cF} \in \textsf{Crys}(\Spec(k),B)$.
    \begin{enumerate}
        \item The perfection $\underline{\widehat{\cF}}$ from Construction~\ref{constr:perf} is naturally isomorphic to the direct summand $\underline{\cF}_{ss} \subset \underline{\cF}$ on which~$\tau$ acts as an isomorphism. In particular, $\underline{\widehat{\cF}} \in \textsf{Coh}_\tau(\Spec(k),B)$. 
        \item $\underline{\cF}$ is flat if and only if the underlying sheaf~$\widehat{\cF}$ of~$\underline{\widehat{\cF}}$ is locally free. \hfill $\Box$
    \end{enumerate}
\end{proposition}



\subsection{\texorpdfstring{$L$}{L}-functions and \texorpdfstring{$l$}{l}-series of crystals}
From now on, we assume that $B$ is a \emph{good coefficient ring} as defined in \cite[Def. 9.7.4]{BP}. For example, $B$ can be artinian or a normal integral domain; in particular, $A = \mathcal O_C(C \setminus \{\infty\})$ is a good coefficient ring. If $\mathfrak p_1,\ldots,\mathfrak p_n$ are the minimal primes of $B$, we denote by $Q_B := B_{\mathfrak p_1} \oplus \cdots \oplus B_{\mathfrak p_n}$ the quotient ring of $B$. Note that the inclusion $B \hookrightarrow Q_B$ is flat, so the functor $- \otimes_B Q_B$ on crystals is exact.

Let $\underline{\mathcal F}$ be a flat $B$-crystal on a scheme $X$. For a closed point $x \in |X|$ of degree $d$, 
denote by $i_x\colon \Spec(\bF_{q^d}) \to X$ the corresponding inclusion. We now recall the $L$-function of a crystal, defined in \cite[Chapter 9]{BP}.

\begin{definition}
Suppose $B$ is an artinian $\bF_q$-algebra, and let $\underline{\mathcal F} \in \textsf{Crys}^\text{flat}(X,B)$. The \emph{crystalline $L$-function} associated to $\underline{\mathcal F}$ is the power series
\[
L( X,\underline{\mathcal F},t) = \prod_{x \in |X|} \det_{B} \left(1 - t\tau \ \bigg{|} \ \widehat{i_x^*\mathcal F}\right)^{-1} \in 1 + tB[\![t]\!].
\]
For general good coefficient rings $B$ and $\underline{\mathcal F} \in \textsf{Crys}^\text{flat}(X,B)$, define
\[
L( X,\underline{\mathcal F},t) = L(X, \underline{\mathcal F} \otimes_B Q_B,t) \in 1 + tB[\![t]\!].
\]
For a derived crystal $\underline{\mathcal F}^\bullet \in D^b(\textsf{Crys}^\text{flat}(X,B))$, define
\[
L(X,\underline{\mathcal F}^\bullet,t) = \prod_{i \in \mathbb Z} L( X,\underline{\mathcal F}^i,t)^{(-1)^i}.
\]
\end{definition}



Our aim is to extend the notion of $L$-functions to crystals on tame Deligne--Mumford stacks. 
To do this without losing the good functoriality properties $L$-functions satisfy, one must take into account the fact that points can have automorphisms.

More precisely, let $\mathfrak X$ be a tame Deligne--Mumford stack of finite type over~$\bF_q$. Then for every $x \in \mathfrak X(\bF_{q^n})$, we have a group scheme~$\underline{\Aut}_x$, given by the following 2-fiber product:
\[
\begin{tikzpicture}
        \node(A){$\underline{\Aut}_x$};
        \node(X) at (0.5,-0.5){\scalebox{1.5}{$\lrcorner$}};
        \node at (2.5,0) (B){$\mathfrak X$};
        \draw[->](A) to (B);
        \node(C)[below of =A]{$\Spec(\bF_{q^n})$};
        \draw[->](A) to (C);
        \node(D)[below of =B]{$\mathfrak X \times \mathfrak X$};
        \draw[->](B) to node{$\Delta$}(D);
        \draw[->](C) to node{$(x,x)$}(D);
    \end{tikzpicture}
\]
The conditions on~$\mathfrak X$ imply that $\underline{\Aut}_x \to \Spec(\bF_{q^n})$ is finite \'etale and the cardinality of the group $\Aut(x) := \underline{\Aut}_x(\bF_{q^n})$ 
is coprime to~$p$. We denote by~$[\mathfrak X(\bF_{q^n})]$ the finite set of isomorphism classes of objects in the groupoid~$\mathfrak X(\bF_{q^n})$.

\begin{lemma}\label{lem:Ltraces}
    Let $\underline{\mathcal F} \in \textsf{Crys}^\text{flat}(X,B)$ with $B$ artinian. Consider the map $\text{d}\log \colon 1 + tB[\![t]\!] \to B[\![t]\!]$ sending $f \mapsto f'/f$. Then we have
    \[
    t \cdot \text{d}\log L(X,\underline{\mathcal F},t) = \sum_{n \geq 1} \sum_{x \in X(\bF_{q^n})} \Tr_{\bF_{q^n} \otimes B}\left(\tau^{n} \ \bigg{|} \ \widehat{x^*\mathcal F} \right)t^{n}.
    \]
\end{lemma}
\begin{proof}
Using the definition of the $L$-function, basic properties of $\text{d}\log$, and the fact that
\[
-t \cdot \text{d}\log \det_R(1-t\psi \ | \ M) = \sum_{n \geq 1} \Tr_R(\psi^n \ | \ M)t^n 
\]
for any endomorphism $\psi$ of a projective $R$-module $M$ (see e.g.\ \cite[Thm 4.7]{gek-sny}), we obtain the formula
\[
t \cdot \text{d}\log L(X,\underline{\mathcal F},t) = \sum_{x \in |X|} \sum_{n \geq 1} \Tr_{B} \left( \tau^{n} \ \bigg{|} \ \widehat{i_x^*\mathcal F}\right) t^{n}.
\]
Next, for a closed point $x \in |X|$ with residue field~$k_x$ of degree~$d$, we have
    \[
    \det_B \left(1 - t\tau \ \bigg{|} \ \widehat{i_x^* \mathcal F} \right) = \det_{k_x \otimes B}\left( 1-t^d\tau^d \ \bigg{|} \ \widehat{i_x^*\mathcal F}\right ),
    \]
by \cite[Lemma 8.1.4]{BP}. Using the chain rule, this gives 
\[
\Tr_B \left(\tau^n \ \bigg{|} \ \widehat{i_x^*\mathcal F}\right) = \begin{cases}
    d \Tr_{k_x \otimes B} \left(\tau^{n} \ \bigg{|} \ \widehat{i_x^*\mathcal F}\right) & \text{ if }d \mid n; \\ 0 & \text{otherwise.}
\end{cases}
\]
Finally, if $d \mid n$, there are precisely $d$ maps $x_1,\ldots,x_d \in X(\bF_{q^n})$ with image~$x$. For any such map~$x_j = i_x \circ \overline{x}_j$, we have
\[
\Tr_{\bF_{q^n} \otimes B}\left(\tau^n \ \bigg{|} \ \widehat{x_j^*\mathcal F}\right) = \Tr_{\bF_{q^n} \otimes B}\left(\tau^n \ \bigg{|} \ \overline{x}_j^* \widehat{i_x^*\mathcal F}\right) = \Tr_{k_x \otimes B}\left(\tau^n \ \bigg{|} \  \widehat{i_x^*\mathcal F} \right).
\]
Combining the above gives the formula from the lemma.
\end{proof}

\begin{definition}\label{def:l-series}
    Suppose $B$ is an artinian $\bF_q$-algebra, and let $\underline{\cF} \in \textsf{Crys}^\text{flat}(\mathfrak X,B)$. The \emph{crystalline $l$-series} associated to~$\underline{\cF}$ is the power series
\[
    l(\mathfrak X,\underline{\mathcal F},t) = \sum_{n \geq 1}\sum_{x \in [\mathfrak X(\bF_{q^n})]}\frac{\Tr_{\bF_{q^n} \otimes B}\left(\mathcal \tau^n \ \bigg{|} \ \widehat{x^*\mathcal F}\right)}{\#\Aut(x)}t^n \in tB[\![t]\!].
\]
For general good coefficient rings~$B$ and $\underline{\mathcal F} \in \textsf{Crys}^\text{flat}(X,B)$, define
\[
l( X,\underline{\mathcal F},t) = l(X, \underline{\mathcal F} \otimes_B Q_B,t) \in tB[\![t]\!].
\]
For a derived crystal $\underline{\mathcal F}^\bullet \in D^b(\textsf{Crys}^\text{flat}(X,B))$, define
\[
l(X,\underline{\mathcal F}^\bullet,t) = \sum_{i \in \mathbb Z} (-1)^i l( X,\underline{\mathcal F}^i,t).
\]
\end{definition}


\begin{remark}
It follows from Lemma \ref{lem:Ltraces} that $l(\mathfrak X,\underline{\mathcal F}^\bullet,t) = t \cdot \text{d}\log L(\mathfrak X,\underline{\mathcal F}^\bullet,t)$ if $\mathfrak X$ is a scheme.
\end{remark}

\begin{lemma}\label{lem:l-additive}
    Given a distinguished triangle $\underline{\cF}^\bullet \to \underline{\mathcal G}^\bullet \to \underline{\mathcal H}^\bullet \to \underline{\cF}^\bullet[1]$ of flat derived crystals on $\mathfrak X$, we have
    \[
    l(\mathfrak X,\underline{\mathcal G}^\bullet,t) = l(\mathfrak X,\underline{\cF}^\bullet,t) + l(\mathfrak X,\underline{\mathcal H}^\bullet,t).
    \]
\end{lemma}

\begin{proof}
    It suffices to prove the statement for artinian $B$ and for a distinguished triangle induced by a short exact sequence of complexes. By the definition of $l$-series, it suffices to prove it in a single degree and at a single point $x \in \mathfrak X(\bF_{q^n})$. In this setting, the statement follows because the perfection functor is exact and traces are additive in short exact sequences.
\end{proof}

\subsection{The trace formula}

One of the main results of \cite{BP} is the trace formula: for a morphism $f\: Y \to X$ of schemes of finite type and $\underline{\mathcal F}^\bullet \in D^b(\textsf{Crys}^\text{flat}(Y,A))$, we have
\[
L(Y,\underline{\mathcal F}^\bullet,t) \sim_{\text{uni}} L(X,Rf_!\underline{\mathcal F}^\bullet,t),
\]
where $\sim_{\text{uni}}$ means that the quotient of these power series is a unipotent polynomial; that is, an element in $P \in 1 + B[t]$ such that $P-1$ is nilpotent. 
We now state the main result of this section.

\begin{theorem}[Lefschetz trace formula]\label{thm:trace_formula}
    Let $B$ be a good coefficient ring. Let $f\: \mathfrak Y \to \mathfrak X$ be a compactifiable morphism of tame Deligne--Mumford stacks of finite type over~$\mathbb F_q$. Then for any $\underline{\mathcal F}^\bullet \in D^b(\textsf{Crys}^\text{flat}(\mathfrak Y,B))$, we have
\[
l(\mathfrak Y,\underline{\mathcal F}^\bullet,t) \sim_{\text{nil}} l(\mathfrak X, Rf_!\underline{\mathcal F}^\bullet,t),
\]
where $\sim_{\text{nil}}$ means that the difference of these power series is~$tP$ for some nilpotent polynomial~$P$. In particular, if~$B$ is reduced, the $l$-series are equal.
\end{theorem} 

We will prove the trace formula by hand for the classifying stack of a finite \'etale group scheme. Combining this with the trace formula for schemes will then imply the theorem in the generality stated above.

\begin{lemma}
    Theorem \ref{thm:trace_formula} holds when $\mathfrak X$ and $\mathfrak Y$ are schemes.
\end{lemma}
\begin{proof}
    By \cite[Thm 9.8.2]{BP}, we may replace~$B$ with~$Q_B$, so in particular we may assume~$B$ is artinian. Thus the lemma follows from \cite[Thm 9.6.5]{BP}, Lemma~\ref{lem:Ltraces}, and the fact that $\text{d}\log P$ is nilpotent if~$P$ is unipotent.
\end{proof}

\begin{proposition}\label{prop:BG}
    Fix $d \geq 1$. Let $G$ be a finite \'etale group scheme over $\Spec(\bF_{q^d})$ of order coprime to~$p$. Let $s \: \mathcal{B}G = [*/G] \to \Spec(\bF_{q^d})$ be the classifying stack of~$G$, and let $\underline{\mathcal F} \in \textsf{Crys}^\text{flat}(\mathcal{B}G,B)$ be a flat $B$-crystal on~$\mathcal{B}G$. Then we have
    \[
    l(\mathcal{B}G,\underline{\mathcal F},t) = l(\Spec(\bF_{q^d}), Rs_!\underline{\mathcal F},t).
    \]
\end{proposition}

\begin{proof}
  By replacing $q$ with $q^d$ and $\tau$ with $\tau^d$, we may assume that $d = 1$.
  We further assume that $B$ is artinian and $\underline{\cF} = \widehat{\underline{\cF}}$.
  The finite \'etale group $G$ is equivalent to an abstract finite group $\bar{G} := G(\overline{\bF}_q)$ with a group automorphism $\sigma \: \bar{G} \to \bar{G}$ induced by the action of $\Gal(\bF_q)$. We extend this correspondence to $G$-torsors as in \cite[Prop.~2.3.2]{behrend_thesis}; see also \cite[\mbox{}7.8--7.10]{olsson_fujiwara}. 
  
  Let $k/\bF_{q}$ be a field extension of degree~$n$. Define a group action of~$\bar{G}$ on itself via $x \star_n g = xg\sigma^n(x^{-1})$. Then $\mathcal{B}G(k)$ is equivalent to the action groupoid for~$\star_n$. Explicitly, the groupoid $\mathcal{B}G(k)$ has as objects the elements of~$\bar{G}$ and as morphisms
\[
    \text{Hom}_{\mathcal{B}G(k)}(g,h) = \{ x \in \bar{G} \ | \ xg\sigma^n(x^{-1}) = h \}.
\]
In particular, $[\mathcal{B}G(k)] = \bar{G}/\star_n$, the set of orbits in $\bar{G}$ under the action $\star_n$. Moreover, $\Aut_{\mathcal{B}G(k)}(g)$ is identified with the stabilizer of~$g$.
  
  Since $\underline{\cF}$ is flat, it is represented by a locally free $\tau$-sheaf on~$\mathcal{B}G$. Applying Proposition~\ref{prop:tau_descent} to the \'etale cover $\Spec(\bF_q) \to \mathcal{B}G$ induced by the trivial torsor, we see that this data is equivalent to a pair $(M,\tau)$, where $M$ is a $B$-module equipped with a $G$-action and $\tau \: \sigma^*M \to M$ is a $G$-linear map, where $G$ acts on $\sigma^*M$ through $\sigma^{-1}$. Let $\bar{M} := \overline{\bF}_q \otimes M$ denote the induced $\bar{G}$-representation, and $\bar{\tau} := \text{id} \otimes \tau$ the $\overline{\bF}_q \otimes B$-linear endomorphism satisfying $\bar{\tau}(gm) = \sigma(g) \bar{\tau}(m)$. If $[g] \in \bar{G}/\star_n$ represents an $\bF_{q^n}$-point of $\mathcal{B}G$, then $\overline{\bF}_q \otimes_{\bF_{q^n}} [g]^*M$ may be identified with $\bar{M}$. The induced endomorphism $\bar{\tau}^n_{[g]}$ of $\bar{M}$ is then conjugate to $\bar{\tau}^n \circ g$ (note that this does not depend on the chosen representative of~$[g]$).
  
  Since $\mathcal{B}G$ is proper, we have $Rs_! = Rs_*$.
  The functor $s_*$ 
  is given by taking $G$-invariants (note that these are respected by~$\tau$).
  Since $\mathcal{B}G$ is tame, $s_*$ is exact.
  We conclude that $Rs_!\underline{\cF} \cong s_*\underline{\cF}$, which is represented by $(M^G,\tau)$.
Now fix $n \geq 1$. We have
    \begin{align*}
        \sum_{x \in [\mathcal{B}G(\bF_{q^n})]}\frac{\Tr_{\bF_{q^n} \otimes B}(\mathcal \tau^n \hspace{0.2em} | \hspace{0.3em} x^*\mathcal F)}{\#\Aut(x)} &= \sum_{[g] \in \bar{G}/\star_n} \frac{\Tr_{\bF_{q^{n}} \otimes B}(\tau^n \hspace{0.2em} | \hspace{0.3em} [g]^*M)}{\# \text{Stab}_{\bar{G}}(g)} \\
        &= \sum_{[g] \in \bar{G}/\star_n} \frac{\Tr_{\overline{\bF}_q \otimes B}(g \circ \bar{\tau}^n \hspace{0.2em} | \hspace{0.3em} \bar{M})}{\# \text{Stab}_{\bar{G}}(g)} \\
        &= \sum_{g \in \bar{G}} \frac{\Tr_{\overline{\bF}_q \otimes B}(g \circ \bar{\tau}^n \hspace{0.2em} | \hspace{0.3em} \bar{M})}{\# \bar{G}} \\
        &= \Tr_{\overline{\bF}_q \otimes B}\left( \frac{1}{\# \bar{G}} \sum_{g \in \bar{G}} g \circ \bar{\tau}^n \ \bigg{|} \ \bar{M} \right) \\
        &= \Tr_{\overline{\bF}_q \otimes B}( \bar{\tau}^n \hspace{0.2em} | \hspace{0.3em} \bar{M}^{\bar{G}}) \\
        &= \Tr_{\bF_{q^n} \otimes B}(\tau^n \hspace{0.2em} | \hspace{0.3em} \bF_{q^n} \otimes M^G),
    \end{align*}
    which verifies the trace formula.
\end{proof}

We now deduce the trace formula as stated in Theorem~\ref{thm:trace_formula}. The following proof is inspired by \cite[\textsection 2]{behrend_thesis}, as well as the treatment~\cite{sun}. It uses noetherian induction on the topological space $|\mathfrak Y| := \mathfrak Y(\overline{\bF}_q) / \Gal(\bF_q)$.

\begin{proof}[Proof of Theorem \ref{thm:trace_formula}]
We apply a sequence of reduction steps.

    1. By Lemma~\ref{lem:l-additive} and the fact that $Rf_!$ preserves distinguished triangles, it suffices to prove the theorem for a single crystal, as we can apply the truncation sequence
    \[
    T_{\leq n}\underline{\cF}^\bullet \longrightarrow \underline{\cF}^\bullet \longrightarrow T_{> n}\underline{\cF}^\bullet \longrightarrow T_{\leq n}\underline{\cF}^\bullet[1]
    \]
    to reduce the statement to complexes of smaller length.

    2. It is enough to show the statement for the structure morphisms $s_{\mathfrak Y}\colon \mathfrak Y \to \Spec(\mathbb F_q)$ for all~$\mathfrak Y$. Indeed, given $f\colon \mathfrak Y \to \mathfrak X$, we have $Rs_{\mathfrak Y,!} \cong Rs_{\mathfrak X,!}Rf_!$ and hence
    \[
    l(\mathfrak Y,\underline{\mathcal F}^\bullet,t) \sim_{\text{nil}} l(\Spec(\mathbb F_q),Rs_{\mathfrak Y,!}\underline{\mathcal F}^\bullet,t) \sim_{\text{nil}} l(\mathfrak X, Rf_!\underline{\mathcal F}^\bullet,t).
    \]
    
    3. By noetherian induction, it is enough to show that the trace formula holds on a non-empty open substack of $\mathfrak Y$, cf.~\cite[Lemma 9.6.8]{BP}. Hence we may assume that $\mathfrak Y$ is a quotient stack; in particular, it admits a coarse moduli scheme $\pi\: \mathfrak Y \to Y$. 
    
    4. It is enough to show the statement for the coarse moduli space~$\pi$. Indeed, we already know the statement for the structure map $s_Y$ since $Y$ is a scheme, and we have $Rs_{\mathfrak Y,!} \cong Rs_{Y,!}R\pi_!$.
    
    5. By proper base change and the local definition of $l$-series, we may take a point $y \in |Y|$ of degree~$d$ and replace~$\pi$ with $\pi_y \colon \mathfrak Y_y \to y$. Then~$\pi_y$ defines an $\underline{\text{Aut}}_y$-gerbe, but any gerbe over a finite field is neutral \cite[Cor.~6.4.2]{behrend_l-adic}.  
    Hence $\pi_y$ is of the form $\mathcal{B}G \to \Spec(\mathbb F_{q^d})$, where $G = \underline{\text{Aut}}_y$ is a finite \'etale group scheme of order coprime to~$p$. Thus we are done by Proposition~\ref{prop:BG}.
\end{proof}
\section{Applications to Drinfeld modular forms}\label{sec:ramanujan}

Our aim in this section is to apply the Lefschetz trace formula to certain crystals on the moduli stack of Drinfeld modules. This will yield a trace formula for Hecke operators on Drinfeld cusp forms, from which we will deduce the Ramanujan bound.

For the remainder of this paper, Drinfeld modules will be assumed to be of rank $r=2$. We briefly recall the definition of Drinfeld modular forms, as well as the Eichler-Shimura theory developed by B\"ockle \cite{bockle}, before proving our main results.

\subsection{Background}

Let $A = \mathcal O_C(C \setminus \{\infty\})$ and $K = \Frac(A)$ as before. Let $K_\infty$ be the completion of $K$ at the place $\infty$, and let $\mathbb C_\infty$ be the completion of an algebraic closure of $K_\infty$. Denote by $\Omega := \mathbb C_\infty \setminus K_\infty$ the Drinfeld upper half-plane, seen as a rigid-analytic space. It has an action of $\operatorname{GL}_2(K_\infty)$ by M\"obius transformations.

A subgroup $\Gamma \subseteq \operatorname{GL}_2(A)$ is called a \emph{congruence subgroup} if it contains
\[
\Gamma(\mathfrak n) = \left \{ 
    M \in \operatorname{GL}_2(A) \ \bigg{|} \ M \equiv \begin{pmatrix} 1 & 0 \\ 0 & 1 
\end{pmatrix} \pmod {\mathfrak{n}} \right \}
\]
for some non-zero ideal $\mathfrak{n} \trianglelefteq A$.

\begin{definition}\label{def:dmf}
    Fix integers $k, l \in \mathbb Z$ and a congruence subgroup $\Gamma$. A \emph{Drinfeld modular form of weight~$k$, type~$l$, and level~$\Gamma$} is a function $f\: \Omega \to \mathbb C_\infty$ satisfying the following properties:
    \begin{enumerate}
        \item For each $\gamma = \begin{pmatrix} a & b \\ c & d \end{pmatrix} \in \Gamma$, we have
        \begin{equation}\label{eq:transf_law}
        f(\gamma z) = \det(\gamma)^{-l} (cz + d)^k f(z).
        \end{equation}
        \item $f$ is holomorphic on $\Omega$ and at the cusps of~$\Gamma$ (see \cite[V.2.4]{gekeler_modcurves}).
    \end{enumerate}
    If additionally $f$ vanishes at all the cusps of~$\Gamma$, we call $f$ a \emph{cusp form}. We denote the space of Drinfeld cusp forms of weight~$k$, type~$l$, and level~$\Gamma$ by~$\S_{k,l}(\Gamma)$.
\end{definition}

\begin{remark}
    By a modular form of level~1, we mean a modular form of level~$\operatorname{GL}_2(A)$. We write~$\S_{k,l}$ for the space of cusp forms of weight~$k$, type~$l$, and level~1.
\end{remark}


\begin{remark}
    In \cite[\textsection 5]{bockle}, B\"ockle defines \emph{adelic Drinfeld modular forms} and shows that spaces of adelic Drinfeld modular forms are naturally isomorphic to spaces of Drinfeld modular forms. The advantage of this is that one can define Hecke operators on Drinfeld modular forms for more general rings~$A$ than just~$\mathbb F_q[T]$. We will use the adelic language where necessary without recalling all details. In particular, we will consider modular forms of level~$\mathcal K$, where $\hat{A} = \varprojlim A/I$ is the completion of~$A$ and $\mathcal K \subset \operatorname{GL}_2(\hat{A})$ is an admissible subgroup. For a prime ideal $\mathfrak p \trianglelefteq A$ not dividing the minimal conductor of~$\mathcal K$, we denote by~$\T_{\mathfrak p}$ the corresponding Hecke operator acting on~$\S_{k,l}(\mathcal K)$.
\end{remark}

\begin{remark}\label{rmk:hecke}
    When $A = \bF_q[T]$, the action of the Hecke operators on adelic cusp forms is different from the Hecke operators defined in \cite{goss_pi-adic,gekeler_coeffs}. For clarity, we distinguish between the two notions of Hecke operators: we retain the notation~$\T_{\mathfrak p}$ for the adelic Hecke operators, and write $\T_{\mathfrak p}^{\mathbb F_q[T]}$ the Hecke operators from loc.\ cit. Let us emphasise an important difference between the two.
    
    From Definition~\ref{def:dmf}, it follows immediately that we have canonical isomorphisms $\S_{k,l} \cong \S_{k,l'}$ for any $l \equiv l' \pmod{q-1}$. This isomorphism is compatible with the Hecke operators $\T_{\mathfrak p}^{\mathbb F_q[T]}$. However, it is not equivariant with respect to the Hecke operators $\T_{\mathfrak p}$: the action differs by a character \cite[Remark 6.12]{bockle}. Given integers $k,l \in \mathbb Z$, we have
\[
\T_{\mathfrak p} = P^{l-k}\T_{\mathfrak p}^{\mathbb F_q[T]},
\]
where $P$ is the monic generator of~$\mathfrak p$. This is why we define $l \in \mathbb Z$ as an integer and not as an element in $\mathbb Z/(q-1)\mathbb Z$.
\end{remark}

\subsection{B\"ockle--Eichler--Shimura theory}

Denote by $\Theta\: \mathfrak M_2 \to \Spec(A)$ the structure morphism of the moduli stack of Drinfeld modules. Fix a maximal ideal $\mathfrak p \trianglelefteq A$ with residue field $\bF_{\mathfrak p}$, corresponding to a map $i_{\mathfrak p} \: \Spec(\bF_{\mathfrak p}) \to \Spec(A)$. Similarly, for the completion $K_\infty$ of $K = \Frac(A)$ at $\infty$, write $i_{K_\infty}\: \Spec(K_\infty) \to \Spec(A)$. 

The \emph{universal Drinfeld module over $\mathfrak M_2$} is the Drinfeld module $\phi_{\text{univ}} := \text{id}\: \mathfrak M_2 \to \mathfrak M_2$. We denote its associated $\tau$-sheaf by $\underline{\cF} := \underline{\mathcal M}(\phi_{\text{univ}})$, cf.~Construction~\ref{constr:ass-tau}.

\begin{definition}
Fix integers $k \geq 2$ and $l \in \mathbb Z$. With notation as above, we define the $A$-crystals
\[
\underline{\cF}_{k,l} := (\det \underline{\cF})^{\otimes l-k+1} \otimes \text{Sym}^{k-2} \underline{\mathcal F}; \qquad \underline{\mathcal S}_{k,l} := R^1\Theta_{!}\underline{\cF}_{k,l}.
\]
We call $\underline{\mathcal S}_{k,l}$ the \emph{crystal of Drinfeld cusp forms of weight~$k$ and type~$l$}.
\end{definition}


Similarly, given an admissible subgroup $\mathcal K \subset \operatorname{GL}_2(\hat{A})$, one has a moduli space $\mathfrak M_{2,\mathcal K}$ of Drinfeld modules with level $\mathcal K$-structure. If $\mathcal K$ has minimal conductor $\mathfrak n \trianglelefteq A$, then the moduli space of Drinfeld modules with level $\mathcal K$-structure is representable by a scheme over $\Spec(A[\mathfrak n^{-1}])$. Forgetting the level structure defines an \'etale cover $\pi_{\mathcal K} \: \mathfrak M_{2,\mathcal K} \to \mathfrak M_2$, and we define
\[
 \underline{\cF}_{k,l}(\mathcal K) := \pi_{\mathcal K}^* \underline{\mathcal F}_{k,l}; \qquad \underline{\mathcal S}_{k,l}(\mathcal K) :=  R^1 \Theta_{\mathcal K,!} \underline{\cF}_{k,l}(\mathcal K).
\]
We call $\underline{\mathcal S}_{k,l}(\mathcal K)$ the \emph{crystal of cusp forms of weight~$k$, type~$l$, and level~$\mathcal K$}. For any prime $\mathfrak p \trianglelefteq A[\mathfrak n^{-1}]$, one has a Hecke operator $\T_{\mathfrak p}$ acting on the crystals of cusp forms via correspondences \cite[Section 13.1]{bockle}.

\begin{lemma}\label{lem:universal-crystals}
    Let $k\geq 2$ and $l \in \mathbb Z$, and let $x \in \mathfrak M_2(\bF_{q^n})$ correspond to the Drinfeld module~$(E,\phi)$.
    \begin{enumerate}
        \item The perfection of $x^* \underline{\cF}_{k,l} \otimes_A K$ is the coherent $\tau$-sheaf
    \[
    \underline{\mathcal M}_{k,l}(\phi)_K := ( (\det \underline{\mathcal M}(\phi))^{\otimes l-k+1} \otimes \Sym^{k-2}\underline{\mathcal M}(\phi)) \otimes_A K.
    \]
    \item The crystals $\underline{\cF}_{k,l}$, $\underline{\cF}_{k,l}(\mathcal K)$, $\underline{\mathcal S}_{k,l}$, and $\underline{\mathcal S}_{k,l}(\mathcal K)$ are flat.
    \end{enumerate}
\end{lemma}

\begin{proof}
    By definition of the universal Drinfeld module, we have $x^* \underline{\cF} \cong \underline{\mathcal M}(\phi)$ as $A$-crystals. By Example~\ref{ex:Mphi}, $\tau$ acts as an isomorphism on $\underline{\mathcal M}(\phi) \otimes_A K$, since $\pi_{\phi} = \tau^n$ is an isogeny and hence becomes invertible after tensoring with~$K$. Hence~$\tau$ also acts as an isomorphism on $\underline{\mathcal M}_{k,l}(\phi)_K$. Since $K$ is artinian, the perfection functor sends a $\tau$-sheaf to the direct summand on which~$\tau$ acts as an isomorphism, which yields~1. For~2, note that $\underline{\mathcal M}(\phi)$ is flat, as its underlying sheaf is locally free. Hence the result follows from Lemma~\ref{lem:flat} and the fact that flatness is preserved under the functors from Theorem~\ref{thm:functors}.
\end{proof}

For any rigid-analytic space~$\mathcal X$, there is a category of \emph{rigid $A$-crystals} which we denote by $\widetilde{\textsf{Crys}}(\mathcal X,A)$. For any scheme $X$ over $\Spec(K_\infty)$, one can define a rigidification functor
\[
(-)^\text{rig}: \textsf{Crys}(X,A) \longrightarrow \widetilde{\textsf{Crys}}(X^\text{rig},A)
\]
where $X^\text{rig}$ is a certain rigid analytic space associated to $X$. 
When applied to the crystal of cusp forms, one can show the following \cite[Cor.~10.13]{bockle}:

\begin{lemma}\label{lem:E-S}
    The crystal $i_{K_{\infty}}^*\underline{\mathcal S}_{k,l}(\mathcal K)^\text{rig}$ is represented by a rigid $\tau$-sheaf of the form
    \[
    \underline{\tilde{\mathds{1}}}_{K_\infty} \otimes_A P_{k,l}(\mathcal K),
    \]
    where $P_{k,l}(\mathcal K)$ is a projective finitely generated $A$-module and $\underline{\tilde{\mathds{1}}}_{K_\infty}$ denotes the trivial rigid $\tau$-sheaf on $\text{Spm}(K_\infty)$. \hfill \qedsymbol{}
\end{lemma}

For a scheme~$X$, the trivial $\tau$-sheaf $\underline{\mathds{1}}_X$ on~$X$ over~$A$ is given by the pair $(\mathcal O_{X \otimes A},\tau)$, where
\[
\tau \colon (\sigma_X \times \text{id})^*\mathcal O_{X \otimes A} \longrightarrow \mathcal O_{X \otimes A}
\]
is given by the adjoint of the map $x \otimes a \mapsto x^q \otimes a$. The construction for rigid $\tau$-sheaves is entirely similar. In particular, the functor taking $\tau$-invariants
\begin{align*}
(-)^\tau \colon \widetilde{\textsf{Crys}}(\mathcal X,A) &\longrightarrow A\textsf{-Mod}\\
\underline{\cF} &\longmapsto H^0(\mathcal X \times \Spec(A),\cF)^\tau,
\end{align*}
which is well-defined on crystals, sends $\underline{\tilde{\mathds{1}}}_{\mathcal X} \otimes_A P \mapsto P$.

We can now state the Eichler--Shimura isomorphism for Drinfeld cusp forms, which is proven in \cite[Thm 10.3]{bockle} in the case $l = k-1$ and which is readily generalised to arbitrary types $l$:

\begin{theorem}[Eichler--Shimura isomorphism]\label{thm:E-S}
Let $\mathcal K \subset \operatorname{GL}_2(\hat{A})$ be an admissible subgroup. Then there is a Hecke-equivariant isomorphism 
\begin{equation}\label{eq:E-S-iso}
\mathbb C_{\infty} \otimes_A \left( i_{K_{\infty}}^*\underline{\mathcal S}_{k,l}(\mathcal K)^\text{rig} \right)^\tau \cong \S_{k,l}(\mathcal K)^\vee,
\end{equation}
called the \emph{Eichler--Shimura isomorphism for Drinfeld cusp forms}. \hfill $\Box$
\end{theorem}


There is also an Eichler--Shimura relation. Classically, this says that for a prime $p \nmid N$, the Hecke operator at~$p$ is the sum of the $p$-Frobenius and the Verschiebung as correspondences on~$X_0(N)$. For Drinfeld modules, the Hecke operator at~$\mathfrak p$ is simply the $\mathfrak p$-Frobenius \cite[Theorem 13.10]{bockle}:

\begin{theorem}[Eichler--Shimura relation]\label{thm:E-S_rel}
    Suppose $\mathcal K$ has minimal conductor~$\mathfrak n$. Let~$\mathfrak p$ be a prime not dividing~$\mathfrak n$. Then we have an equality
    \begin{equation}\label{eq:E-S-rel}
    i_{\mathfrak p}^*\T_{\mathfrak p} = \tau^{\deg(\mathfrak p)}
    \end{equation}
    as endomorphisms of $i_{\mathfrak p}^*\underline{\mathcal S}_{k,l}(\mathcal K)$. \hfill $\Box$
\end{theorem}

The above results extend to the crystals $\underline{\mathcal S}_{k,l}$ of level~1. In particular, we obtain the following:

\begin{corollary}\label{cor:BES}
    Denote by $\underline{\hat{\mathcal S}}^{K,\mathfrak p}_{k,l}$ the perfection of $i_{\mathfrak p}^*\underline{\mathcal S}_{k,l} \otimes_A K$. Then we have
    \[
    \Tr_{\bF_{\mathfrak p} \otimes K} \left( \tau^{\deg (\mathfrak p)n} \ \bigg{|} \ \hat{\mathcal S}^{K,\mathfrak p}_{k,l} \right) = \Tr_{\mathbb C_\infty} \left( \T_{\mathfrak p}^n \hspace{0.3em} \bigg{|} \hspace{0.1em} \S_{k,l} \right).
    \]
\end{corollary}

\begin{proof}
    Consider $\mathcal K = \mathcal K(\mathfrak n)$ for some proper non-zero ideal $\mathfrak n$ not contained in $\mathfrak p$. Then $\mathcal K(\mathfrak n) \subset \operatorname{GL}_2(\hat{A})$ is a normal subgroup with quotient $G \cong \operatorname{GL}_2(A/\mathfrak n A)$, and as a result $\mathfrak M_{2,\mathcal K}$ is a Galois cover of $\mathfrak M_{2}$ over $\Spec(A[\mathfrak n^{-1}])$ with Galois group~$G$. In other words, we can identify $\mathfrak M_2$ with the stack quotient $[\mathfrak M_{2,\mathcal K}/G]$. Via base change, this statement holds also over both $K_\infty$ and~$\bF_{\mathfrak p}$.

    Since a $G$-action on a crystal induces by functoriality a $G$-action on its rigidification, we define $i_{K_\infty}^*\underline{\mathcal S}_{k,l}^\text{rig} := (i_{K_\infty}^*\underline{\mathcal S}_{k,l}(\mathcal K)^\text{rig})^G$. Then the Eichler--Shimura isomorphism~\eqref{eq:E-S-iso} and Lemma~\ref{lem:E-S} imply
    \[
    i_{K_\infty}^*\underline{\mathcal S}_{k,l}^\text{rig} \cong (\underline{\mathds{1}}_{K_\infty} \otimes_A \S_{k,l}(\mathcal K)^\vee)^G \cong \underline{\mathds{1}}_{K_\infty} \otimes_A \S_{k,l}^\vee.
    \]
    Next, $i_{\mathfrak p}^*\T_{\mathfrak p}^n = \tau^{\deg (\mathfrak p)n}$ follows from the Eichler--Shimura relation~\eqref{eq:E-S-rel} and Corollary~\ref{cor:crys_pullback}, since exact and conservative functors are faithful. Combining the above and noting that $A$ is reduced, the Hecke action on~$\S_{k,l}$ can be computed on the underlying sheaf of any locally free representative of the crystal~$i_{\mathfrak p}^*\underline{\mathcal S}_{k,l}$. This may also be done after tensoring with $K$, which gives the desired expression.
\end{proof}

\subsection{A trace formula for Hecke operators}

Let $\pi \in \bar{K}$ be algebraic over~$K$ of degree $\leq 2$, and denote its Galois conjugate by~$\bar{\pi}$. Then we set for $m \geq 0$,
\[
\Tr_m(\pi) := \sum_{i = 0}^m \pi^i \bar{\pi}^{m-i}.
\]

Combining Corollary~\ref{cor:BES} with the Lefschetz trace formula, we obtain a trace formula for Hecke operators on Drinfeld cusp forms of level~1.

\begin{theorem}\label{thm:heckesum}
    Let $\mathfrak p \trianglelefteq A$ be a maximal ideal with residue field $\bF_{\mathfrak p}$. Fix a weight $k \geq 2$ and a type $l \in \mathbb Z$. Then for every $n \geq 1$, we have
    \begin{equation}\label{eq:hecketrace}
    \Tr_{\mathbb C_\infty}(\T^n_{\mathfrak p} \hspace{0.2em} | \hspace{0.1em} \S_{k,l}) = \sum_{[(E,\phi)]/\bF_{{\mathfrak p}^n}} \Tr_{k-2}(\pi_{\phi}) \cdot (\pi_\phi \bar{\pi}_\phi)^{l-k+1},
    \end{equation}
    where the sum is taken over the set of isomorphism classes of Drinfeld modules over $\bF_{\mathfrak{p}^n}$, and $\pi_{\phi} \in \bar{K}$ denotes the Frobenius endomorphism of $(E,\phi)$.
\end{theorem}

\begin{proof}
    Let $j_{\mathfrak p} \: \mathfrak M_{2,\mathfrak p} \to \mathfrak M_2$ be the fiber of the moduli space of Drinfeld modules at~$\mathfrak p$. Consider the $l$-series of the crystal $j_{\mathfrak p}^*\underline{\cF}_{k,l}$. 
    By Lemma~\ref{lem:universal-crystals}, we have
    \[
    l(\mathfrak M_{2,\mathfrak p},j_{\mathfrak p}^*\underline{\cF}_{k,l},t) = \sum_{n \geq 1} \sum_{[(E,\phi)]/\bF_{\mathfrak p^n}} \frac{\Tr_{\bF_{\mathfrak{p}^n} \otimes K}\left(\tau^{\deg(\mathfrak p)n} \ \bigg{|} \ \underline{\mathcal M}_{k,l}(\phi)_K \right)}{\# \Aut(\phi)} t^{\deg(\mathfrak p)n}.
    \]
    By Lemma~\ref{lem:aut_order}, we have $\# \Aut(\phi) = -1$ for all~$\phi$. Moreover, the endomorphism $\tau^{\deg(\mathfrak p)n}$ of the $\tau$-sheaf associated to~$\phi$ may be identified with the Frobenius endomorphism $\pi_{\phi}$ of~$\phi$.
    If the characteristic polynomial of~$\pi_{\phi}$ has roots~$\pi_{\phi}$ and~$\bar{\pi}_\phi$, then by definition of $\underline{\mathcal M}_{k,l}(\phi)_K$ we obtain
    \[
        l(\mathfrak M_{2,\mathfrak p},j_{\mathfrak p}^*\underline{\cF}_{k,l},t) = \sum_{n \geq 1} \sum_{[(E,\phi)]/\bF_{\mathfrak p^n}} - (\pi_\phi \bar{\pi}_\phi)^{l-k+1} \Tr_{k-2}(\pi_{\phi}) t^{\deg(\mathfrak p)n}.
    \]
    On the other hand, since $\mathfrak M_{2,\mathfrak p}$ is compactifiable \cite[Theorem~F]{rydh_tame_comp}, we may apply the trace formula to the crystal $j_{\mathfrak p}^*\underline{\cF}_{k,l}$ under the structure map $s_{\mathfrak p} \: \mathfrak M_{2,\mathfrak p} \to \Spec(\bF_{\mathfrak p})$. We have $R^i\Theta_!\underline{\cF}_{k,l} = 0$ for $i \neq 1$ by \cite[Theorem 8.4.2(a)]{BP} and \'etale descent. Combined with proper base change, this gives $Rs_{\mathfrak p,!}j_{\mathfrak p}^*\underline{\cF}_{k,l} \cong i_{\mathfrak p}^*\underline{\mathcal S}_{k,l}[-1]$. Thus the trace formula states that
    \[
    l(\mathfrak M_{2,\mathfrak p},j_{\mathfrak p}^*\underline{\cF}_{k,l},t) = l(\Spec(\bF_{\mathfrak p}),i_{\mathfrak p}^*\underline{\mathcal S}_{k,l}[-1],t),
    \]
    with equality holding on the nose since $A$ is reduced. Applying Corollary~\ref{cor:BES} and comparing coefficients yields the desired equation~\eqref{eq:hecketrace}.
\end{proof}

\begin{remark}\label{rmk:level}
    Let $\mathcal K \subset \operatorname{GL}_2(\hat{A})$ be an admissible subgroup. Applying the Lefschetz trace formula to the crystal $\underline{\mathcal S}_{k,l}(\mathcal K)$ yields a similar formula for traces of Hecke operators on Drinfeld cusp forms of level $\mathcal K$: for any $\mathfrak p$ not dividing the minimal conductor of~$\mathcal K$, we have
    \[
\Tr_{\mathbb C_\infty}(\T^n_{\mathfrak p} \hspace{0.2em} | \hspace{0.1em} \S_{k,l}(\mathcal K)) = -\sum_{[(E,\phi,[\psi])]/\bF_{{\mathfrak p}^n}} \Tr_{k-2}(\pi_{\phi}) \cdot (\pi_\phi \bar{\pi}_\phi)^{l-k+1},
    \]
    where the sum is taken over the set of isomorphism classes of Drinfeld modules over $\mathbb F_{{\mathfrak p}^n}$ with level $\mathcal K$-structure.
\end{remark}

\subsection{Ramanujan bounds}

As a consequence of the trace formula, we obtain bounds on the traces of Hecke operators. Denote by $|\cdot|_\infty$ a fixed extension to $\mathbb C_\infty$ of the norm on $K$ given by $|x|_\infty = q^{-\deg(\infty)v_\infty(x)}$.

\begin{corollary}[Ramanujan bound]\label{cor:rambound_tr}
    Let $\mathfrak p \trianglelefteq A$ be a maximal ideal of degree~$d$. Then for every $n \geq 1$, we have
    \begin{equation}\label{eq:ram_bound}
    |\Tr_{\mathbb C_\infty}(\T_{\mathfrak p}^n \hspace{0.2em} | \hspace{0.1em} \S_{k,l})|_{\infty} \leq (q^{nd})^{\frac{k}{2} + l - k}.
    \end{equation}
\end{corollary}

\begin{proof}
For any Drinfeld module over~$\bF_{\mathfrak p^n}$ with Frobenius endomorphism~$\pi$, the elements~$\pi$ and~$\bar{\pi}$ are Weil numbers of rank~$2$ over $\mathbb F_{\mathfrak p^n} \cong \mathbb F_{q^{nd}}$. In particular,
    \[
    |\pi|_{\infty} = |\bar{\pi}|_{\infty} = q^\frac{nd}{2}.
    \]
By equation~\eqref{eq:hecketrace}, the definition of~$\Tr_{k-2}$, and the triangle inequality, we obtain
    \[
    q^{-nd(l-k+1)}|\Tr_{\mathbb C_\infty}(\T_{\mathfrak p}^n \hspace{0.2em} | \hspace{0.1em} \S_{k,l})|_\infty \leq \max_\pi \{ | \Tr_{k-2}(\pi,\bar{\pi})|_\infty \} \leq q^{nd(k-2)/2}.
    \]
    This yields the desired inequality.
\end{proof}

\begin{remark}
    Using Remark~\ref{rmk:level}, one can use the same argument to obtain a Ramanujan bound for the traces of~$\T_{\mathfrak p}$ acting on $\S_{k,l}(\mathcal K)$ for any~$\mathfrak p$ not dividing the minimal conductor of~$\mathcal K$.
\end{remark}

\begin{remark}
    The bound~\eqref{eq:ram_bound} involves the exponent $k/2 + l - k$, which at first sight looks strange. We argue that the main term in the exponent is~$k/2$, whereas the term $l-k$ is an artefact of the adelic interpretation of Hecke operators. As explained in Remark~\ref{rmk:hecke}, it is natural to view~$l$ as an element in $\mathbb Z/(q-1)\mathbb Z$, but then the adelic Hecke operator is only defined up to a character twist. After a suitable normalisation, the exponent does indeed reduce to $k/2$, as demonstrated by the case $A = \mathbb F_q[T]$ (cf.\ Corollary~\ref{cor:rambound2} below).
\end{remark}

\begin{corollary}[Ramanujan bound for {$\bF_q[T]$}]\label{cor:rambound2}
    Let $A = \bF_q[T]$ and fix a monic and irreducible polynomial $P \in A$ of degree~$d$. Let $\mathfrak p = (P)$ and denote by $\T_{\mathfrak p}^{\mathbb F_q[T]}$ the Hecke operator for $\mathfrak p$ as defined in \cite[\textsection 7]{gekeler_coeffs}. Then for every $n \geq 1$, we have
    \[
    \deg  \Tr \left( \left(\T_{\mathfrak p}^{\mathbb F_q[T]}\right)^n \hspace{0.1em} \bigg{|} \hspace{0.2em} \S_{k,l} \right) \leq \frac{ndk}{2}.
    \]
\end{corollary}

\begin{proof}
    This follows because $|f|_\infty = q^{\deg(f)}$ for $f \in A$, and the fact that $\T_{\mathfrak p}^{\mathbb F_q[T]} = P^{k-l}\T_{\mathfrak p}$.
\end{proof}

We conclude with an observation about the Hecke eigenvalues. In characteristic~0, Newton's identities imply that for an operator $\T$ acting on a $d$-dimensional vector space~$V$, knowing $\Tr(\T^n)$ for $n = 1, \ldots, d$ is equivalent to knowing the eigenvalues of~$\T$. Although this fails in characteristic~$p$, one can still say the following.

\begin{corollary}
    Suppose the action of $\T_{\mathfrak p}$ on $\S_{k,l}$ does not have~$p$ repeated eigenvalues. Then any eigenvalue $\alpha$ of $\T_{\mathfrak p}$ satisfies 
    \[
    |\alpha|_{\infty} \leq (q^d)^{\frac{k}{2} + l - k}.
    \]
\end{corollary}

\begin{proof}
    This is a direct consequence of the Ramanujan bound~\eqref{eq:ram_bound} and \cite[Prop.~4.2]{devries_newton}.
\end{proof}

We further study the consequences of the trace formula for $A = \bF_q[T]$ in \cite{devries_traces}.

\begingroup
\sloppy
\printbibliography

@book{BP,
    author = {B\"ockle, Gebhard and Pink, Richard},
    year = {2009},
    title = {Cohomological Theory of Crystals over Function Fields},
    publisher = {European Mathematical Society},
    series = {Tracts in Mathematics},
    edition = {1},
    volume = {9},
}

@book{gekeler_modcurves,
    author = {Gekeler, Ernst-Ulrich},
    year = {1986},
    title = {Drinfeld Modular Curves},
    publisher = {Springer},
    series = {Lecture Notes in Mathematics},
    volume = {1231},
}

@book{goss_book,
    author = {Goss, David},
    year = {1996},
    title = {Basic Structures of Function Field Arithmetic},
    publisher = {Springer},
    series = {Ergebnisse der Mathematik und ihrer Grenzgebiete},
    volume = {35},
}

@book{laumon,
  author = {Laumon, G\'erard},
  year = {1996},
  title = {Cohomology of Drinfeld Modular Varieties, Part I},
  publisher = {Cambridge University Press},
  series = {Cambridge studies in advanced mathematics},
  volume = {41},
    }

@book{olsson,
    author = {Olsson, Martin},
    year = {2016},
    title = {Algebraic Spaces and Stacks},
    publisher = {American Mathematical Society},
    series = {Colloquium Publications},
    volume = {62},
    }

@book{papikian,
    author = {Papikian, Mihran},
    year = {2023},
    title = {Drinfeld Modules},
    publisher = {Springer},
    series = {Graduate Texts in Mathematics},
    edition = {1},
    volume = {296},
}

@article{band-val_atkin,
    author = {Bandini, Andrea and Valentino, Maria},
    title={On the {A}tkin ${U}_t$-operator for ${\Gamma}_0(t)$-invariant {D}rinfeld cusp forms},
    journal={Proceedings of the American Mathematical Society},
    volume = {147},
    number = {10},
    pages = {4171–4187},
    year={2019},
}

@article{behrend_l-adic,
    author = {Behrend, Kai},
    year = {2003},
    pages = {},
    title = {Derived $\ell$-adic categories for algebraic stacks},
    volume = {774},
    journal = {Memoirs of the American Mathematical Society},
}

@phdthesis{behrend_thesis,
    author = {Behrend, Kai},
    title = {The Lefschetz Trace Formula for the Moduli Stack of Principal Bundles},
    school = {University of California, Berkeley},
    year = {1991},
    url = {https://personal.math.ubc.ca/~behrend/thesis.pdf},
}

@misc{bockle,
    author = {B\"ockle, Gebhard},
    year = {2002},
    title = {An Eichler-Shimura isomorphism over function fields between Drinfeld modular forms and cohomology classes of crystals},
    note = {Preprint},
    url = {https://www1.iwr.uni-heidelberg.de/fileadmin/groups/arithgeo/templates/data/Gebhard_Boeckle/EiShNew.pdf},
    }

@article{deligne_ramanujan,
author = {Deligne, Pierre},
journal = {S\'eminaire Bourbaki},
pages = {139-172},
publisher = {Springer-Verlag},
title = {Formes modulaires et repr\'esentations $l$-adiques},
volume = {11},
year = {1969},
}

@article{deligne_weil,
author = {Deligne, Pierre},
title = {La conjecture de Weil : I},
journal = {Publications Math\'ematiques de l'IH\'ES},
volume = {43},
year = {1974},
pages = {273-307},
}

@article{drinf_ell,
    author = {Drinfel'd, Vladimir},
    title = {Elliptic modules},
    journal = {Mat. Sb. (N.S.)},
    year = {1974},
    volume = {94(136)},
    issue = {4(8)},
    pages = {594-627},
}

@article{drinf_comm-subrings,
    author = {Drinfel'd, Vladimir},
    title = {Commutative subrings of certain noncommutative rings},
    journal = {Functional Analysis and Its Applications},
    year = {1977},
    volume = {11},
    issue = {1},
    pages = {9-12},
}

@article{gekeler_coeffs,
    author = {Gekeler, Ernst-Ulrich},
    title = {On the Coefficients of Drinfeld modular forms},
    journal = {Inventiones Mathematicae},
    year = {1988},
    volume = {93},
    pages = {667-700},
}

@incollection{gek-sny,
    author = {Gekeler, Ernst-Ulrich and Snyder, Brian},
    title = {Drinfeld Modules Over Finite Fields},
    booktitle = {Drinfeld Modules, Modular Schemes and Applications},
    publisher = {World Scientific},
    year = {1997},
}

@article{goss_pi-adic,
    author = {Goss, David},
    title = {$\pi$-adic Eisenstein series for function fields},
    journal = {Compositio Mathematica},
    year = {1980},
    volume = {41},
    pages = {3-38},
}

@article{li-meemark,
    author = {Li, Wen-Ching Winnie and Meemark, Yotsanan},
    title = {Hecke operators on Drinfeld cusp forms},
    journal = {Journal of Number Theory},
    year = {2008},
    volume = {128},
    pages = {1941-1965},
}

@phdthesis{mann_thesis,
    author = {Mann, Lucas},
    title = {A $p$-Adic 6-Functor Formalism in Rigid-Analytic Geometry},
    school = {Rheinische Friedrich-Wilhelms-Universität Bonn},
    year = {2022},
    archivePrefix = {arXiv},
    eprint = {2206.02022},
}

@article{nicole-rosso,
    author = {Nicole, Marc-Hubert and Rosso, Giovanni},
    title = {Perfectoid {D}rinfeld Modular Forms},
    journal = {Journal de Th\'eorie des Nombres de Bordeaux},
    year = {2021},
    volume = {33},
    pages = {1045-1067},
}

@article{olsson_fujiwara,
    author = {Olsson, Martin},
    title = {{F}ujiwara's theorem for equivariant correspondences},
    journal = {Journal of Algebraic Geometry},
    volume = {24},
    year = {2015},
}

@unpublished{rydh_tame_comp,
    author = {Rydh, David},
    title = {Compactification of tame Deligne–Mumford stacks},
    url = {https://people.kth.se/~dary/papers.html},
    year = {2011},
    note = {Preprint},
    note = {Accessed: 23 October 2023}
    }

@article{sun,
    author = {Sun, Shenghao},
    title = {$L$-series of Artin stacks over finite fields},
    journal = {Algebra \& Number Theory},
    volume = {6},
    year = {2012},
    pages = {47-122},
}

@article{devries_newton,
    author = {de Vries, Sjoerd},
    title = {On {N}ewton's identities in positive characteristic},
    year = {2025},
    journal = {Journal of Algebra},
    volume = {668},
    pages = {348-364},
}

@unpublished{devries_traces,
    author = {de Vries, Sjoerd},
    title = {Traces of Hecke operators on Drinfeld modular forms for \texorpdfstring{$\mathbb{F}_q[T]$}{Fq[T]}},
    year = {2024},
    archivePrefix = {arXiv},
    eprint = {2407.04555},
}

@thesis{diva,
    author = {de Vries, Sjoerd},
    title = {Traces of Hecke operators on Drinfeld modular forms via point counts},
    type = {Licentiate thesis},
    school = {Stockholm University},
    year = {2023}, 
    url = {https://urn.kb.se/resolve?urn=urn:nbn:se:su:diva-223837},
}

@misc{stacks_project,
    shorthand    = {Stacks},
    author       = {The {Stacks Project Authors}},
    title        = {\textit{The Stacks Project}},
    year         = {2024},
    url          = {https://stacks.math.columbia.edu},
  }
\endgroup
\end{document}